\documentclass[11pt,reqno]{amsart}
\usepackage{amscd,amssymb,amsmath,amsthm}
\usepackage[arrow,matrix]{xy}
\usepackage{graphicx}
\usepackage{epstopdf}
\usepackage{color}
\newtheorem{thm}{Theorem}
\newtheorem{defn}{Definition}
\newtheorem{lemma}{Lemma}
\newtheorem{pro}{Proposition}
\newtheorem{rk}{Remark}
\newtheorem{cor}{Corollary}

\numberwithin{equation}{section} 

\begin{document}
\title [A dynamical system of temperature-dependent sex linked inheritance]
{A dynamical system of temperature-dependent sex linked
inheritance}

\author{Z.S. Boxonov,  U.A. Rozikov}

\address{Z.\ S.\ Boxonov\\ Namangan state University, Namangan, Uzbeksiatn.}
 \email {z.b.x.k@mail.ru}

 \address{U.\ A.\ Rozikov \\ Institute of mathematics, 29, Do'rmon Yo'li str., 100125,
Tashkent, Uzbekistan.} \email {rozikovu@yandex.ru}
\begin{abstract} Recently, R.Varro introduced a gonosomal algebra of 
the temperature-dependent sex determination system which is controlled 
by three temperature ranges. In this paper we study dynamical systems which are 
given by quadratic evolution operators of the gonosomal algebras of sex-liked populations. 
We show that this evolution operator can be reduced to an evolution operator of free population.
Then using behavior of the free population
 we describe the set of limit points for trajectories of several evolution operators  
of the sex-linked populations.      
\end{abstract}


 \maketitle
\section{Introduction}

It is known\footnote{see
https://en.wikipedia.org/wiki/Sex-determination$_-$system} that
some species of reptiles, including alligators, some turtles, and
the tuatara, sex is determined by the temperature at which the egg
is incubated during a temperature-sensitive period. In reptiles
(snakes, crocodiles, turtles, lizards) there are two types of sex
determination, either a genotypic determination controlled
according to the species by XY- or ZW-system, either a
determination depending on the incubation temperature of eggs.

Following \cite{V} we recall that temperature-dependent sex
determination is controlled by three temperature ranges, the eggs
subject to feminizing temperatures (resp. masculinising) give rise
to 100 percentage  or a majority of females and those subject to
transition temperatures provide 50 percentage females and 50
percentage males.

In \cite{V} the following algebraic model is considered: let $A$
be a linear space with basis $(e_{i},\tilde{e_{i}})_{1\leq i\leq
n}$ where $e_{i}$
 (resp. $\tilde{e_{i}})$ are female (resp. male) genetic types
present in a population and subject to temperature-dependent sex
determination. Denote by $\tau_{1},$  $\tau_{2}$  and $ \tau_{3}$
the probability that eggs are incubated at feminizing
temperatures, masculinizing temperatures and transition
temperatures respectively, thus
$$\tau_{1}, \tau_{2}, \tau_{3}\geq 0, \ \ \tau_{1}+\tau_{2}+\tau_{3}=1.$$
 For each $r=1, 2$ ($r=1$ for feminizing temperatures, $r=2$ for
masculinizing temperatures) denote by $\mu_{r}$ and
$\tilde{\mu_{r}}$ respectively the proportions of females and
males arising from eggs placed in the environment $r$, thus we
have
$$\mu_{r} + \tilde{\mu_{r}} = 1, \ \ \mu_{1} \geq
\tilde{\mu_{1}} \geq 0, \ \ 0 \leq \mu_{2} \leq \tilde{\mu_{2}}.$$
For all $1 \leq i, p \leq n$ denote by $\theta_{ipk}$ the egg
proportion of $e_{k}$ type in the laying of a female $e_{i}$
crossed with a male $\tilde{e_{p}}$, thus
\begin{equation}\label{the}
\theta_{ipk}\geq 0, \ \   \sum_{k=1}^{n}\theta_{ipk}=1.
\end{equation}
 Then the space $A$ equipped with the
following algebra structure:
$$e_{i} e_{j} = \tilde{e_{p}} \tilde{e_{q}} = 0,$$
\begin{equation}\label{2.1}e_{i}
\tilde{e_{p}}=(\mu_{1}\tau_{1}+\mu_{2}\tau_{2}+\frac{1}{2}\tau_{3})\sum_{k=1}^{n}\theta_{ipk}e_{k}+
(\tilde{\mu_{1}}\tau_{1}+\tilde{\mu_{2}}\tau_{2}+\frac{1}{2}\tau_{3})\sum_{k=1}^{n}\theta_{ipk}\tilde{e_{k}}.
\end{equation}
This algebra $A$ is called a gonosomal algebra (see \cite{V}) and
the product $e_{i}\tilde{e_{p}}$ gives the genetic distribution of
progeny of a female $e_{i}$ with a male $\tilde{e_{p}}$.

Denote
$$a=\mu_{1}\tau_{1}+\mu_{2}\tau_{2}+\frac{1}{2}\tau_{3}, \ \ b=\tilde{\mu_{1}}\tau_{1}+\tilde{\mu_{2}}\tau_{2}+\frac{1}{2}\tau_{3}.$$
Note that $a\geq 0$, $b\geq 0$ and $a+b=1$.

Consider the following set\footnote{This set is a subset of
$(2n-1)$-dimensional simplex.}
$$S=\left\{z=(x_{1},...,x_{n},y_{1},...,y_{n})\in{R}^{2n},\right.$$ $$
\left.x_{i}\geq0, y_{j}\geq0, \sum_{i=1}^{n}x_{i}\neq0,
\sum_{j=1}^{n}y_{j}\neq0,
\sum_{i=1}^{n}x_{i}+\sum_{j=1}^{n}y_{j}=1\right\}.$$

We call the partition into types hereditary if for each possible
state $z\in S$ describing the current generation, the state $z'\in
S$ is uniquely defined describing the next generation. This means
that the association $z\rightarrow z'$ defines a map
$W:S\rightarrow S$ called the evolution operator \cite{L}.

For any point $z^{(0)}\in S$ the sequence $z^{(t)}=W(z^{(t-1)}),
t=1,2,...$ is called the trajectory of $z^{(0)}$. Denote by $\omega(z^{(0)})$ the set of
limit points of the trajectory. Since $\{z^{(n)}\}\subset S$
and $S$ is compact, it follows that $\omega(z^{(0)})\ne
\emptyset.$ Obviously, if $\omega(z^{(0)})$ consists of a single point,
then the trajectory converges, and $\omega(z^{(0)})$ is a fixed point.
However, looking ahead, we remark that convergence of the
trajectories is not the typical case for the dynamical systems.

If $z'=(x_{1}',...,x_{n}',y_{1}',...,y_{n}')$ is a state of the
system gens in the next generation then by the rule (\ref{2.1}) we
get the evolution operator $W_a:S\rightarrow S$ defined by

\begin{equation}\label{2.2}W_a:\left\{
\begin{array}{ll}
    x_{k}'=\frac{a \sum_{i=1}^{n} \sum_{p=1}^{n} \theta_{ipk}x_{i}y_{p}}{ \sum_{i=1}^{n} x_{i} \sum_{p=1}^{n} y_{p}},  \\[3mm]
    y_{k}'=\frac{(1-a) \sum_{i=1}^{n} \sum_{p=1}^{n} \theta_{ipk}x_{i}y_{p}} { \sum_{i=1}^{n} x_{i} \sum_{p=1}^{n} y_{p}},
\end{array}\right.
\end{equation} where $k=1,...,n$.

In this paper our goal is to study dynamical systems generated by operator (\ref{2.2}). 
 
\section{Dynamics systems generated by the operator (\ref{2.2})}
For $a\in (0,1)$ we denote
$$S_a\equiv S^{2n-2}_a=\left\{z=(x_{1},...,x_{n},y_{1},...,y_{n})\in S: \ \sum_{i=1}^{n}x_{i}=a, \sum_{j=1}^{n}y_{j}=1-a\right\}.$$
The following lemma is useful
\begin{lemma}\label{l1}
\begin{itemize} For any fixed $a\in (0,1)$ we have
\item[1.] for any $z=(x,y)\in S$ the following holds
$$z'=(x',y')=W_a(z)\in S_a, \ \ \mbox{i.e.}, \ \ W_a(S)\subset S_a;$$
\item[2.] the set $S_a$ is invariant with respect
to $W_a$, i.e., $W_a(S_a)\subset S_a$.
\end{itemize}
\end{lemma}
\begin{proof}
The proof follows from the following easily checked equality
\begin{equation}\label{ab}
\sum_{k=1}^nx_k'=a, \ \ \ \sum_{k=1}^ny_k'=1-a.
\end{equation}
\end{proof}
For each fixed $a\in (0,1)$,  by this Lemma \ref{l1}, the
investigation of the sequence $z^{(t)}=W_a(z^{(t-1)})$,
$t=1,2,...$, for each point $z^{(0)}\in S$ is reduced to the case
$z^{(0)}\in S_a$. Therefore we are interested to the following
dynamical system:

$$z^{(0)}\in S_a, z^{(1)}=W_a(z^{(0)}), z^{(2)}=W_a(z^{(1)}), \dots$$
the main problem is to study the limit
\begin{equation}\label{lim}
\lim_{m\to\infty}z^{(m)}=\lim_{m\to\infty}W_a^m(z^{(0)}).
\end{equation}
 The restriction on $S_a$ of the operator $W_a$, denoted simply by $W$, has the form
 \begin{equation}\label{2.2a}
 W:\left\{%
\begin{array}{ll}
    x_{k}'=(1-a)^{-1} \sum_{i=1}^{n} \sum_{p=1}^{n} \theta_{ipk}x_{i}y_{p},  \\[3mm]
    y_{k}'=a^{-1} \sum_{i=1}^{n} \sum_{p=1}^{n} \theta_{ipk}x_{i}y_{p},  \\
\end{array}%
\right.\end{equation} where $k=1,...,n$.

Denote $\beta=(1-a)/a$. A point $z\in S_a$ is called a fixed point
of $W$ if $W(z)=z$.
\begin{lemma}\label{l2}
If $z=(x,y)\in S_a$ is a fixed point of $W$ then
$$y_k=\beta x_k, \ \ k=1,\dots, n.$$
\end{lemma}
\begin{proof} Straightforward.
\end{proof}
By this lemma the problem of finding fixed points of $W_a$ is
reduced to solution of the following system:
\begin{equation}\label{2.2b}
    x_{k}=a^{-1} \sum_{i=1}^{n} \sum_{p=1}^{n} \theta_{ipk}x_{i}x_{p}, \ \ k=1,\dots,n.
\end{equation}
\begin{lemma}\label{l2n}
If $z^{(0)}=(x_1^{(0)},\dots, x_n^{(0)}; y_1^{(0)},\dots,
y_n^{(0)})\in S_a$ is an initial point and
$z^{(m)}=(x_1^{(m)},\dots, x_n^{(m)}; y_1^{(m)},\dots,
y_n^{(m)})=W^m_a(z^{(0)})$ then
$$y^{(m)}_k=\beta x^{(m)}_k, \ \ \mbox{for any} \ \ k=1,\dots, n, \ \ m\geq 1.$$
\end{lemma}
\begin{proof} Follows from the following equality
\begin{equation}\label{2.m}\left\{%
\begin{array}{ll}
    x_{k}^{(m)}=(1-a)^{-1} \sum_{i=1}^{n} \sum_{p=1}^{n} \theta_{ipk}x^{(m-1)}_{i}y^{(m-1)}_{p},  \\[3mm]
    y_{k}^{(m)}=a^{-1} \sum_{i=1}^{n} \sum_{p=1}^{n} \theta_{ipk}x^{(m-1)}_{i}y^{(m-1)}_{p},  \\
\end{array}%
\right.  k=1,\dots, n, \ \ m\geq 1. \end{equation}
\end{proof}
By this lemma to investigate $z^{(m)}$ it suffices to study
$x^{(m)}=(x_1^{(m)}, \dots, x_n^{(m)})$ given by $x^{(m)}=\tilde V(x^{(m-1)})$, $m\geq 1$, where
\begin{equation}\label{uz}
  \tilde V: \,  x_{k}'=a^{-1} \sum_{i=1}^{n} \sum_{p=1}^{n} \theta_{ipk}x_{i}x_{p}, \ \ k=1, \dots, n.
\end{equation}

Introduce new variables $u_i=x_i/a$, $i=1,2,\dots,m$ then
from (\ref{uz}) we obtain the operator $V$
given by
\begin{equation}\label{kso}
V:\ \  u'_k=\sum_{i=1}^{n} \sum_{p=1}^{n} \theta_{ipk}u_{i}u_{p}, \ \ k=1,\dots,n.
\end{equation}
where $\theta_{ipk}$ satisfies (\ref{the}).

Note that dynamical systems generated by the operator (\ref{kso})
were studied in many papers (see \cite{L}-\cite{RS}, and \cite{GMR} for a review,
\cite{JLM} and references therein for recent results).

Thus we have the following
\begin{cor}\label{c1} The dynamical system given by (\ref{2.2a}) is equivalent to the dynamical system
given by (\ref{kso}).
\end{cor}

To illustrate the above-mentioned results for investigation of (\ref{2.2a}) we consider the following
three classes of operators:

\begin{itemize}
\item[C1.] Assume $\theta_{ipk}=0$ if $k\notin \{i,j\}$, and $\theta_{ipk}\ne {1\over 2}$, for each $i,p,k=1,\dots,n$.

\item[C2.] Divide the set $E=\{1,2,\dots,n\}$ into three parts:
$$\{1\}, F=\{2,\dots,m\}, M=\{m+1,\dots,n\},$$
where $2\leq m\leq n-1$.

 Now define coefficients $\theta_{ijk}$ as
follows
\begin{equation}\label{thet}
\theta_{ijk}=\left\{\begin{array}{lll}
1, \ \ {\rm if} \ \ k=1, i,j\in F\cup \{1\} \ \ {\rm or} \ \ i,j\in
M\cup \{1\};\\
0, \ \ {\rm if} \ \ k\ne 1, i,j\in F\cup \{1\} \ \ {\rm or} \ \
i,j\in M\cup \{1\};\\
\geq 0, \ \ {\rm if } \ \ i\in F, j\in M, \forall k.
\end{array}\right.
\end{equation}

\item[C3.] Now we give an example of operator (\ref{2.2a}) which has periodic orbits.
Consider $n=3$ and coefficients as the following
 \begin{equation}\label{ex}
 \theta_{113}=\theta_{123}=\theta_{213}=\theta_{132}=\theta_{213}=\theta_{223}=\theta_{312}=\theta_{231}=1, \ \ \theta_{331}=c, \theta_{332}=d,
 \end{equation}
 where $c+d=1$ and the remaining $\theta_{ijk}=0$.
 The corresponding operator (\ref{kso}) has the form
 $$
 \left\{\begin{array}{lll} x'_1=cx^2_3+ 2x_2x_3\\
x'_2=dx^2_3+ 2x_1x_3\\
x'_3= (x_1+x_2)^2.\\
\end{array}\right.
$$
\end{itemize}

Denote
$${\rm int} S_a=\left\{z=(x_{1},...,x_{n},y_{1},...,y_{n})\in S_a: \prod_{i=1}^nx_{i}y_i\ne 0\right\}.$$
$$\partial S_a=S_a\setminus {\rm int}S_a.$$

\begin{thm}\label{t1} If condition {\rm C1} is satisfied then
\begin{itemize}
\item[1)] If $z^{(0)}\in {\rm int}S^{m-1}$
is not a fixed point, then
$\omega(z^{(0)})\subset \partial S_a$.

\item[2)] The set $\omega(z^{(0)})$ either consists of a single point or
is infinite.

\item[3)] If the operator (\ref{2.2a}) has an isolated fixed point $z^*\in {\rm int}
S_a$ then for any initial point
$z^{(0)}\in {\rm
int}S_a\setminus \{z^*\}$ the trajectory $\{z^{(n)}\}$ does not
converge.\\

If condition {\rm C2} is satisfied then

\item[4)]  Let $\theta_{ijk}$ be as in (\ref{thet}) then operator (\ref{2.2a})
has a unique fixed point
$((a,0,...,0),(1-a,0,...,0))$. For any $z^0\in S_a$, the
trajectory $\{z^{(n)}\}$ tends to this fixed point exponentially
rapidly.\\

In case {\rm C3} we have

\item[5)] If (\ref{ex}) satisfied then corresponding operator
 (\ref{2.2a}) has unique fixed point
 $$z^*=(ax^*_1, ax^*_2, ax^*_3, (1-a)x^*_1, (1-a)x^*_2, (1-a)x^*_3),$$ where
$$x^*_1={(7-3\sqrt{5})c+4\sqrt{5}-8\over 2(4-\sqrt{5})};\ \
x^*_2={(3\sqrt{5}-7)c+\sqrt{5}-1\over 2(4-\sqrt{5})};\ \
x^*_3={3-\sqrt{5}\over 2}.$$
There is unique cyclic theory $\{(ac,ad,0, 1-a,0,0),(a,0,0,1-a,0,0)\}.$ For any
initial point $z^{(0)}\in S_a$ the $\omega$-limit set has the following form
$$\omega(z^{(0)})=
\left\{\begin{array}{ll}
\{(ac,ad,0, 1-a,0,0),(a,0,0,1-a,0,0)\}, \ \ \mbox {if} \ \ z^{(0)}_3\ne ax^*_3,\\
\{z^*\}, \ \ \ \ \mbox {if} \ \ z^{(0)}_3= ax^*_3.
\end{array}\right.
$$
 \end{itemize}
\end{thm}
\begin{proof} 1)-3) follow from Theorem 2.4 of \cite{GMR} by Corollary \ref{c1}.

Part 4) follows from Theorem 2.24 of \cite{GMR} (see also \cite{RJ}).

Part 5) is the consequence of Corollary \ref{c1} and Theorem 2 of \cite{RZ}.
 \end{proof}
\begin{rk} It follows from Theorem \ref{t1} that the
operator (\ref{2.2a}) has a trajectory which
converges; or does not converge with a finite set of limit points;
or does not converge with an infinite set of limit points. Note that
these are all possible cases which one expects for a given sequence of real numbers.
Therefore, we can choose parameters of the dynamical system generated by the operator (\ref{2.2a})
to have as rich behavior as needed.
\end{rk}

\section{Full analysis in two dimensional case}

In general it is difficult to solve the system (\ref{2.2b}). Let
us solve it for $n=2$.

{\it Case $n=2$}: In this case using $x_2=a-x_1$ we can reduce the
first equation to a quadratic equation with four parameters, which
has the following two solutions
\begin{equation}\label{qe}\begin{array}{ll}
x_{1,1}={a\over 2}\cdot{2\theta_3-\theta_2+1+\sqrt{(\theta_2-1)^2+4\theta_3(1-\theta_1)}
\over \theta_1+\theta_3-\theta_2},\\[3mm]
x_{1,2}={a\over 2}\cdot{2\theta_3-\theta_2+1-\sqrt{(\theta_2-1)^2+4\theta_3(1-\theta_1)}
\over \theta_1+\theta_3-\theta_2},
\end{array}
\end{equation}
where
$$\theta_1=\theta_{111},\ \ \theta_2=\theta_{121}+\theta_{211}, \ \ \theta_3=\theta_{221}.$$

 Thus if the parameters satisfy the following condition
\begin{equation}\label{con}
(\theta_2-1)^2+4\theta_3(1-\theta_1)\geq 0
\end{equation}
 then the fixed points are
$$z_{j}=(x_{1,j},x_{2,j}, y_{1,j}, y_{2,j}), j=1,2,$$
where $x_{2,i}=a-x_{1,i}$ and $y_{i,j}=\beta x_{i,j}$.
Below we find conditions on parameters of the operator (\ref{uz}) to guarantee  $z_{j}\in S_a$.

For $n=2$, using $x_{1}+x_{2}=a$, from (\ref{uz}) we get
\begin{equation}\label{T}
T:x_{1}'=a^{-1}(\theta_1+\theta_3-\theta_2)x_{1}^{2}+(\theta_2-2\theta_3)x_{1}+\theta_3a
\end{equation}
Note that (\ref{T}) maps $S_{a}^{1}=[0, \,a]$ to itself, i.e., $T:S_{a}^{1}\mapsto S_{a}^{1}$.

\begin{lemma}\label{zafar}
\begin{itemize}
\item[\emph{a}.] Uniqueness of fixed point:

\item[1)] If $\theta_{1}\in[0;1),
\theta_{2}\in[0;1], \theta_{3}=0$ then (\ref{T}) has a unique fixed point $x_{1}=0$;

\item[2)] If $\theta_{1}=0, \theta_{2}=1, \theta_{3}=1$ then (\ref{T})
has a unique fixed point $x_{1}=\frac{a}{2}$;

\item[3)] If $\theta_{1}=1, \theta_{2}\in[1;2], \theta_{3}\in(0;1]$ then the fixed point of (\ref{T}) is $x_{1}=a$;

\item[4)] If $\theta_{1}\in[0;1), \theta_{2}\in[0;2], \theta_{3}\in(0;1]$
then the fixed point of (\ref{T}) is
 $$x_{1}={a\over
2}\cdot{2\theta_3-\theta_2+1-\sqrt{(\theta_2-1)^2+4\theta_3(1-\theta_1)}
\over \theta_1+\theta_3-\theta_2}.$$

\item[\emph{b}.] Two fixed points:

\item[1)] If $\theta_{1}=1, \theta_{2}\in[0;1), \theta_{3}\in(0;1]$ then
the mapping (\ref{T}) has two fixed points
$$x_{1,1}=a, \ \ x_{1,2}=\frac{\theta_{3}a}{\theta_{3}-\theta_{2}+1};$$

\item[2)] If $\theta_{1}=1, \theta_{2}\in[0;1)\cup(1;2), \theta_{3}=0$ then
(\ref{T}) has two fixed points
$$x_{1,1}=0, \ \ x_{1,2}=a;$$

\item[3)] If $\theta_{1}\in[0;1], \theta_{2}\in(1;2], \theta_{3}=0$ then
mapping (\ref{T}) has fixed points
$$x_{1,1}=0, \, x_{1,2}=\frac{(1-\theta_{2})a}{\theta_{1}-\theta_{2}}.$$

\item[\emph{c}.] If $\theta_1=\theta_2=1, \theta_3=0$ then the set  $S_{a}^1$
is the set of fixed points of (\ref{T}).
\end{itemize}
\end{lemma}
\begin{proof} The proof consists detailed analysis of the quadratic equation
\begin{equation}\label{Tp}
x=a^{-1}(\theta_1+\theta_3-\theta_2)x^{2}+(\theta_2-2\theta_3)x+\theta_3a.
\end{equation}
This equation has solutions (\ref{qe}). Carefully checking of $x_{1,1}, x_{1,2}\in  S_{a}^{1}$
completes the proof.
\end{proof}

For the mapping (\ref{T}) define type of fixed points.
\begin{defn}
Suppose $x_{0}$ is a fixed point for $T$. Then $x_{0}$ is an
attracting fixed point if $|T'(x_{0})|<1$. The point $x_{0}$ is
a repelling fixed point if $|T'(x_{0})|>1$. Finally, if
$|T'(x_{0})|=1$, the fixed point is called neutral or
saddle \cite{D}.
\end{defn}
The following results are very simple to prove:
\begin{lemma}\label{lz}
\begin{itemize}
\item[1)] The type of the unique fixed points for (\ref{T}) are as follows:
$$x_{1}=
\left\{\begin{array}{ll}
    0 \ \hbox{is attracting} \ \ \hbox{if}\ \ \theta_{1}\in[0;1), \theta_{2}\in[0;1), \theta_{3}=0, \\[2mm]
    0 \ \ \hbox{is saddle} \ \ \hbox{if}\ \ \theta_{1}\in[0;1), \theta_{2}=1, \theta_{3}=0, \\[2mm]
    \frac{a}{2} \ \ \hbox{is saddle} \ \ \hbox{if}\ \ \theta_{1}=0, \theta_{2}=1, \theta_{3}=1,  \\
    a \ \ \hbox{is attracting} \ \ \hbox{if}\ \ \theta_{1}=1, \theta_{2}\in(1;2), \theta_{3}\in(0;1], \\[2mm]
    a \ \ \hbox{is saddle} \ \ \hbox{if}\ \ \theta_{1}=1, \theta_{2}\in\{1;2\}, \theta_{3}\in(0;1], \\[2mm]
    x^{*} \ \ \hbox{is attracting} \ \ \hbox{if}\ \ (\theta_2-1)^2+4\theta_3(1-\theta_1)<4, \\[2mm]
    x^{*} \ \ \hbox{is reppeler } \ \ \hbox{if}\ \ (\theta_2-1)^2+4\theta_3(1-\theta_1)>4,  \\[2mm]
    x^{*} \ \ \hbox{is saddle } \ \ \hbox{if}\ \ (\theta_2-1)^2+4\theta_3(1-\theta_1)=4, \\[2mm]
\end{array}\right.$$
where $$x^{*}={a\over
2}\cdot{2\theta_3-\theta_2+1-\sqrt{(\theta_2-1)^2+4\theta_3(1-\theta_1)}
\over \theta_1+\theta_3-\theta_2}.$$
\item[2)] The type of two fixed point of (\ref{T})
$$x_{1,1}=
\left\{\begin{array}{ll}
    a \ \ \hbox{is repelling} \ \ \hbox{if}\ \ \theta_{1}=1, \theta_{2}\in[0;1), \theta_{3}\in(0;1], \\[2mm]
    0 \ \ \hbox{is attarcting } \ \ \hbox{if}\ \ \theta_{1}=1, \theta_{2}\in[0;1), \theta_{3}=0, \\[2mm]
    0 \ \ \hbox{is repelling } \ \ \hbox{if}\ \ \theta_{1}=1, \theta_{2}\in(1;2), \theta_{3}=0,  \\
    \ \ \ \ \ \ \ \ \ \ \ \ \ \ \ \ \ \hbox{or} \ \ \theta_{1}\in[0;1], \theta_{2}\in(1;2], \theta_{3}=0, \\[2mm]
   \end{array}\right.$$
$$x_{1,2}=
\left\{\begin{array}{ll}
    \frac{\theta_{3}a}{\theta_{3}-\theta_{2}+1} \ \ \hbox{is attracting} \ \ \hbox{if}\ \ \theta_{1}=1, \theta_{2}\in[0;1), \theta_{3}\in(0;1], \\[2mm]
    a \ \ \hbox{is repelling} \ \ \hbox{if}\ \ \theta_{1}=1, \theta_{2}\in[0;1), \theta_{3}=0, \\[2mm]
    a \ \ \hbox{is attarcting} \ \ \hbox{if}\ \ \theta_{1}=1, \theta_{2}\in(1;2), \theta_{3}=0,  \\
    \frac{(1-\theta_{2})a}{\theta_{1}-\theta_{2}} \ \ \hbox{is attracting} \ \ \hbox{if}\ \ \theta_{1}\in[0;1], \theta_{2}\in(1;2], \theta_{3}=0. \\[2mm]
   \end{array}\right.$$
\end{itemize}
\end{lemma}

\begin{pro}\label{pro1}
\begin{itemize}
 Let $x^{(0)}\in S_{a}^{1}$ be an initial point then
\item[1)]$$\lim_{m\longrightarrow \infty}T^{(m)}(x_{0})=\lim_{m\longrightarrow \infty}x^{(m)}=
\left\{\begin{array}{ll}
    0, \ \ \hbox{if} \ \ \theta_{1}\in[0;1), \theta_{2}\in[0;1], \theta_{3}=0, \\[2mm]
    a, \ \ \hbox{if} \ \ \theta_{1}=1, \theta_{2}\in[1;2], \theta_{3}\in(0;1],  \\
    x^{*}, \ \ \hbox{if} \ \ \theta_{1}\in[0;1), \theta_{2}\in[0;2],    \theta_{3}\in(0;1], \\
    \frac{\theta_{3}a}{\theta_{3}-\theta_{2}+1}, \ \ \hbox{if} \ \ \theta_{1}=1, \theta_{2}\in[0;1), \theta_{3}\in(0;1], \\[2mm]
    0, \ \ \hbox{if} \ \ \theta_{1}=1, \theta_{2}\in[0;1), \theta_{3}=0,  \\
    a, \ \ \hbox{if} \ \ \theta_{1}=1, \theta_{2}\in(1;2), \theta_{3}=0,  \\
    \frac{(1-\theta_{2})a}{\theta_{1}-\theta_{2}}, \ \ \hbox{if} \ \ \theta_{1}\in[0;1], \theta_{2}\in(1;2], \theta_{3}=0,
\end{array}\right.$$
\item[2)]$$\lim_{k\longrightarrow \infty}T^{(2k)}(x_{0})=\lim_{k\longrightarrow \infty}x^{(2k)}=x^{(0)}, \ \ \hbox{if} \ \ \theta_{1}=0, \theta_{2}=1,
\theta_{3}=1.$$
\item[3)] $$\lim_{k\longrightarrow \infty}T^{(2k+1)}(x_{0})=\lim_{k\longrightarrow \infty}x^{(2k+1)}=a-x^{(0)}, \ \ \hbox{if} \ \ \theta_{1}=0, \theta_{2}=1, \theta_{3}=1.$$
\end{itemize}
\end{pro}
\begin{proof}
Follows from Lemma \ref{lz} taking into account the graph of the mapping (\ref{T}).
\end{proof}
\begin{thm}
\begin{itemize}
 Let $z^{0}=(x^{(0)}_{1},x^{(0)}_{2},y^{(0)}_{1},y^{(0)}_{2})\in S_{a}^{2}$ be an initial point then
\item[1)]$$\lim_{m\longrightarrow \infty}W^{(m)}(z_{0})=$$
$$
\left\{\begin{array}{ccc}
    (0,a,0,1-a),& \hbox{if} & \theta_{1}\in[0;1), \theta_{2}\in[0;1], \theta_{3}=0, \\[2mm]
    (a,0,1-a,0), & \hbox{if} & \theta_{1}=1, \theta_{2}\in[1;2], \theta_{3}\in(0;1],  \\
    (x^{*},a-x^{*},\beta x^{*},\beta (a-x^{*})), & \hbox{if} & \theta_{1}\in[0;1), \theta_{2}\in[0;2],
    \theta_{3}\in(0;1],\\
    (\frac{\theta_{3}a}{\theta_{3}-\theta_{2}+1},\frac{(1-\theta_{2})a}{\theta_{3}-\theta_{2}+1},
    \frac{\theta_{3}(1-a)}{\theta_{3}-\theta_{2}+1},\frac{(1-\theta_{2})(1-a)}{\theta_{3}-\theta_{2}+1}),& \hbox{if} & \theta_{1}=1, \theta_{2}\in[0;1), \theta_{3}\in(0;1], \\[2mm]
    (0,a,0,1-a), & \hbox{if} & \theta_{1}=1, \theta_{2}\in[0;1), \theta_{3}=0,  \\
    (a,0,1-a,0), & \hbox{if} & \theta_{1}=1, \theta_{2}\in(1;2), \theta_{3}=0,  \\
    (\frac{(1-\theta_{2})a}{\theta_{1}-\theta_{2}},\frac{(\theta_{1}-1)a}
    {\theta_{1}-\theta_{2}},\frac{(1-\theta_{2})(1-a)}{\theta_{1}-\theta_{2}},
    \frac{(\theta_{1}-1)(1-a)}{\theta_{1}-\theta_{2}}), & \hbox{if} & \theta_{1}\in[0;1], \theta_{2}\in(1;2], \theta_{3}=0,
\end{array}\right.$$
\item[2)]$$\lim_{k\longrightarrow \infty}W^{(2k)}(z^{(0)})=(x^{(0)}_{1},a-x^{(0)}_{1},\beta x^{(0)}_{1},\beta (a-x^{(0)}_{1})), \ \ \hbox{if} \ \ \theta_{1}=0, \theta_{2}=1,
\theta_{3}=1,$$
$$\lim_{k\longrightarrow
\infty}W^{(2k+1)}(z^{(0)})=(a-x^{(0)}_{1},x^{(0)}_{1},\beta
(a-x^{(0)}_{1}),\beta x^{(0)}_{1}), \ \ \hbox{if} \ \ \theta_{1}=0,
\theta_{2}=1, \theta_{3}=1.$$
\end{itemize}
\end{thm}
\begin{proof} Follows from Lemma \ref{l1}, Lemma \ref{l2} and Proposition \ref{pro1}.
\end{proof}

\section{An example of $n$-dimensional case}

Now we consider the following constraint on heredity coefficients
(\ref{the})
$$\theta_{ipk}=
\left\{\begin{array}{ll}
    c_{ik}, \ \ \hbox{if} \ \ p=j, \\[2mm]
    1, \ \ \hbox{if} \ \ p\neq j, k=l, \\
    0, \ \ \hbox{if} \ \ p\neq j, k\neq l,
\end{array}\right. \ \ j, \, l=1,\dots, n.$$
Then operator defined by (\ref{uz}) has the following form
\begin{equation}\label{4.a}
U_{a}:\left\{%
\begin{array}{ll}
    x_{k}'=a^{-1}\sum_{i=1}^{n}c_{ik}x_{i}x_{j}, \ \ k=1,\dots, n, \ \ k\neq l  \\[3mm]
    x_{l}'=a^{-1}(\sum_{i=1}^{n}c_{il}x_{i}x_{j}+\sum_{i=1}^{n}\sum_{j\neq p=1}^{n}x_{i}x_{p}).  \\[3mm]
    \end{array}%
\right.\end{equation}

We will consider the following case
$$c_{ik}=
\left\{\begin{array}{ll}
    1, \ \ \hbox{} \ \ i=k, \\[2mm]
    0, \ \ \hbox{} \ \ i\neq k,  \\
\end{array}\right.\eqno$$
then the operator defined by (\ref{4.a}) has the following form
\begin{equation}\label{4.b}
U:\left\{%
\begin{array}{ll}
    x_{k}'=a^{-1}x_{k}x_{j}, \ \ k=1,\dots, n, \ \ k\neq l  \\[3mm]
    x_{l}'=a^{-1}(x_{l}x_{j}+\sum_{i=1}^{n}\sum_{j\neq p=1}^{n}x_{i}x_{p}). \\[3mm]
    \end{array}%
\right.\end{equation}

Let us find all fixed points of $U$ given by (\ref{4.b}), i.e. we
solve the following system of equations
\begin{equation} \label{fix}
\left\{%
\begin{array}{ll}
    x_{k}=a^{-1}x_{k}x_{j}, k=1,\dots, n, k\neq l  \\[3mm]
    x_{l}=a^{-1}(x_{l}x_{j}+\sum_{i=1}^{n}\sum_{j\neq p=1}^{n}x_{i}x_{p}), \\[3mm]
    \end{array}%
\right.\end{equation}
From the first equation of the system (\ref{fix}) by $k=j$ and $k\neq l$ we get
 $x_{j,1}=0, x_{j,2}=a$. Consequently, since $\sum_{i=1}^{n}x_{k}=a$ we obtain $x_{l,1}=a, x_{l,2}=0$.
 Thus we have proved the following

\begin{pro}\begin{itemize}
\item[1)] If $j=l$ then the operator  (\ref{4.b}) has unique fixed point
$${\bf x}=(\underbrace{0,...,0,a}_{j},\underbrace{0,...0}_{n-j});$$
\item[2)] If $j \neq l$ then (\ref{4.b}) has two fixed points
$$
{\bf x}_{1,l}=(\underbrace{0,...,0,a}_{l},\underbrace{0,...0}_{n-l}), \ \
{\bf x}_{2,j}=(\underbrace{0,...,0,a}_{j},\underbrace{0,...0}_{n-j}).$$
\end{itemize}
\end{pro}

\begin{pro}\begin{itemize}\label{pro3} Let
$x^{(0)}=(x^{(0)}_{1},...,x^{(0)}_{n})\in S_{a}^{n-1}$ be an
initial point.

\item[1)]If $j=l$ then $$\lim_{m\longrightarrow
\infty}U^{m}(x^{(0)})=\lim_{m\longrightarrow
\infty}x^{(m)}={\bf x}=(\underbrace{0,...,0,a}_{j},\underbrace{0,...0}_{n-j})$$

\item[2)]If $j\neq l$ then
$$\lim_{m\longrightarrow
\infty}U^{m}(x^{(0)})=\lim_{m\longrightarrow \infty}x^{(m)}=
\left\{\begin{array}{ll}
    {\bf x}_{1,l}, \ \ \hbox{if} \ \ x^{(0)}_{j}\neq a, \\[2mm]
    {\bf x}_{2,j}, \ \ \hbox{if} \ \ x^{(0)}_{j}=a
\end{array}\right.$$
\end{itemize}
\end{pro}
\begin{proof}\begin{itemize}
\item[1)] Let $j=l$.  From (\ref{4.b})  we get
\begin{equation}\label{4.c}
U^{(m+1)}(x^{(0)}):\left\{%
\begin{array}{ll}
    x_{k}^{(m+1)}=a^{-1}x_{k}^{(m)}x_{l}^{(m)}, \ \ k=1,\dots, n, \ \ k\neq l    \\[3mm]
    x_{l}^{(m+1)}=a^{-1}(x_{l}^{(m)}x_{l}^{(m)}+\sum_{i=1}^{n}\sum_{l\neq p=1}^{n}x_{i}^{(m)}x_{p}^{(m)}).  \\[3mm]
    \end{array}%
\right.\end{equation}

Since $\sum_{i=1}^{n}x_{i}^{(m)}=a$, $x_{l}^{(m)}\in [0,a]$ we
have
$$x_{l}^{(m+1)}=a^{-1}(x_{l}^{(m)}x_{l}^{(m)}+\sum_{i=1}^{n}\sum_{l\neq p=1}^{n}x_{i}^{(m)}x_{p}^{(m)})=$$ $$a-a^{-1}x_{l}^{(m)}(a-x_{l}^{(m)})\geq a-a^{-1}a(a-x_{l}^{(m)})=x_{l}^{(m)},$$
$$x_{k}^{(m+1)}=a^{-1}x_{k}^{(m)}x_{l}^{(m)}\leq a^{-1}x_{k}^{(m)}a
= x_{k}^{(m)}, \ \ k=1,\dots, n, \ \ k\neq l.$$

Thus $x_{l}^{(m)}$ is a non-decreasing sequence, which bounded
from above by $a$ the sequence $x_{k}^{(m)}$, $k=1, \dots, n$,
$k\neq l$ is a non-increasing and with lower  bound 0.
Consequently, each $x_{i}^{(m)}$ has a limit say $\alpha_{i}$, $i=1,2,....n.$

From equations of (\ref{4.c}) for limit values $\alpha_{1}$ we get
the following equations
\begin{equation}\label{4.d}
\left\{%
\begin{array}{ll}
    \alpha_{l}=a-a^{-1}\alpha_{l}(a-\alpha_{l}),  \\[3mm]
    \alpha_{k}=a^{-1}\alpha_{k}\alpha_{l}, \ \ k=1,\dots, n, \ \ k\neq l.\\[3mm]
    \sum_{i=1}^{n}\alpha_{i}=a
\end{array}%
\right.\end{equation}

It is easy to see that the system (\ref{4.d}) has the following
solution:

$$\alpha_{l}=a,\ \ \alpha_{k}=0,\ \ k=1,\dots, n, \ \ k\neq l.$$

\item[2)] Let $j\neq l$.  From (\ref{4.b})  we get
\begin{equation}\label{4.e}
U^{(m+1)}(x^{(0)}):\left\{%
\begin{array}{ll}
    x_{k}^{(m+1)}=a(a^{-2}x_{k}^{(0)}x_{j}^{(0)})^{2^{m}}, \ \ k=1,\dots, n,\ \ k\neq l    \\[3mm]
    x_{l}^{(m+1)}=a(1-(a^{-1}x_{j}^{(0)})^{2^{m}}(a-x_{l}^{(0)})x_{j}^{-(0)}).  \\[3mm]
\end{array}%
\right.\end{equation} It is easy to see that the system
(\ref{4.e}) has the following limit:

if $x_{j}^{(0)}=a$ then
$$\lim_{m\longrightarrow
\infty}U^{(m+1)}(x^{(0)})=\lim_{m\longrightarrow
\infty}x^{(m+1)}_{k}= \left\{\begin{array}{ll}
    a, \ \ \hbox{if} \ \ k=j, \\[2mm]
    0, \ \ \hbox{if} \ \ k=1,\dots, n, \ \ k\neq j
\end{array}\right.$$
if $x_{j}^{(0)}\neq a$ then
$$\lim_{m\longrightarrow
\infty}U^{(m+1)}(x^{(0)})=\lim_{m\longrightarrow
\infty}x^{(m+1)}_{k}= \left\{\begin{array}{ll}
    a, \ \ \hbox{if} \ \ k=l, \\[2mm]
    0, \ \ \hbox{if} \ \ k=1,\dots, n, \ \ k\neq l
\end{array}\right.$$
\end{itemize}\end{proof}

Summarizing we obtain the following
\begin{thm}\begin{itemize} Let
$z^{(0)}=(x^{(0)}_{1},...,x^{(0)}_{n};y^{(0)}_{1},...,y^{(0)}_{n})\in
S_{a}^{2n-2}$ be an initial point.

\item[1)]If $j=l$ then $$\lim_{m\longrightarrow
\infty}W^{(m)}(x^{(0)})=\lim_{m\longrightarrow
\infty}z^{(m)}=(\underbrace{0,...,0,a}_{j},\underbrace{0,...0}_{n-j};\underbrace{0,...,0,1-a}_{j},\underbrace{0,...,0}_{n-j})$$

\item[2)]If $j\neq l$ then
$$\lim_{m\longrightarrow
\infty}W^{(m)}(x^{(0)})=\lim_{m\longrightarrow \infty}z^{(m)}=$$ $$
\left\{\begin{array}{ll}
    z_{1,l}=(\underbrace{0,...,0,a}_{l},\underbrace{0,...0}_{n-l};\underbrace{0,...,0,1-a}_{l},\underbrace{0,...,0}_{n-l}), \ \ \hbox{if} \ \ x^{(0)}_{j}\neq a, \\[2mm]
    z_{2,j}=(\underbrace{0,...,0,a}_{j},\underbrace{0,...0}_{n-j};\underbrace{0,...,0,1-a}_{j},\underbrace{0,...,0}_{n-j}), \ \ \hbox{if} \ \ x^{(0)}_{j}=a,
\end{array}\right.$$
\end{itemize}
\end{thm}

\section*{ Acknowledgements}

This work was partially supported  by Agencia Estatal de Investigaci\'on (Spain), grant MTM2016-
79661-P (European FEDER support included, UE)

\end{document}